\documentclass{article}

\usepackage[final,nonatbib]{mlwg_2019}

\usepackage[utf8]{inputenc} 
\usepackage[T1]{fontenc}    
\usepackage{hyperref}       
\usepackage{url}            
\usepackage{booktabs}       
\usepackage{amsfonts}       
\usepackage{nicefrac}       
\usepackage{microtype}      
\usepackage{hyperref}
\usepackage{graphicx}
\usepackage[english]{babel}
\usepackage{amssymb,amsmath,amsthm}
\usepackage{thmtools}
\usepackage{thm-restate}
\usepackage{cleveref}

\declaretheorem[name=Lemma]{lemma}
\allowdisplaybreaks

\title{Implicit Neural Solver for Time-dependent Linear PDEs with Convergence Guarantee}

\author{%
  Suprosanna Shit\\
  Technical University Munich\\
  \texttt{suprosanna.shit@tum.de}\\
  \And
  Abinav R.\\
  Technical University Munich\\
  \texttt{abinav.ravi@tum.de}\\
  \And
  Ivan Ezhov\\
  Technical University Munich\\
  \texttt{ivan.ezhov@tum.de}\\
  \And
  Jana Lipkova\\
  Technical University Munich\\
  \texttt{jana.lipkova@tum.de}\\
  \And
  Marie Piraud\\
  Technical University Munich\\
  \texttt{marie.piraud@tum.de}\\
  \And
  Bjoern Menze\\
  Technical University Munich\\
  \texttt{bjoern.menze@tum.de}
}

\begin{document}
\maketitle
\vspace{-1em}
\begin{abstract}
\vspace{-0.5em}
Fast and accurate solution of time-dependent partial differential equations (PDEs) is of key interest in many research fields including physics, engineering, and biology. Generally, implicit schemes are preferred over the explicit ones for better stability and correctness. The existing implicit schemes are usually iterative and employ a general-purpose solver which may be sub-optimal for a specific class of PDEs. In this paper, we propose a neural solver to learn an optimal iterative scheme for a class of PDEs, in a data-driven fashion. We attain this objective by modifying an iteration of an existing semi-implicit solver using a deep neural network. Further, we prove theoretically that our approach preserves the correctness and convergence guarantees provided by the existing iterative-solvers. We also demonstrate that our model generalizes to a different parameter setting than the one seen during training and achieves faster convergence compared to the semi-implicit schemes.
\end{abstract}
\vspace{-1em}
\section{Introduction}
\vspace{-0.5em}
Time-dependent partial differential equations (PDEs) are an essential mathematical tool to describe numerous physical processes such as heat transfer, wave propagation, quantum transport, and tumor growth. Solving the initial-value problem (IVP) and boundary-value problem (BVP) accurately and efficiently is of primary research interest in computational science. Numerical solution of time-dependant PDEs requires appropriate spatio-temporal discretization. Spatial discretization can be cast as either finite difference method (FDM) or finite element method (FEM) or finite volume method (FVM), whereas temporal discretization relies on either explicit, implicit or semi-implicit methods. Explicit temporal update rules are generally a single or few forward computation steps, while implicit or semi-implicit update rules, such as Crank-Nicolson's scheme, resort to a fixed-point iterative scheme. Small time and spatial resolution facilitate a more accurate solution, however, it increases computational burden at the same time. Moreover, maximally allowed spatio-temporal resolution is not only constrained by the desired accuracy but also limited to numerical stability criteria. Note that implicit and semi-implicit methods offer a relaxed stability constraints (sometimes unconditionally stable) at the expense of an increased computational cost caused by the iterative schemes.

In recent times, deep neural networks \cite{raissi2019physics} have gained significant attention by numerical computation community due to its superior performance in solving forward simulations \cite{magill2018neural} and inverse-problems \cite{ezhov2019neural}. Recent work by Tompson et al. \cite{tompson2017accelerating} shows a data-driven convolutional neural network (CNN) can accelerate fluid simulation compared to traditional numerical schemes. Long et al. \cite{long2017pde} shows that learned differential operators can outperform hand-designed discrete schemes. However, on the contrary to the well understood and theoretically grounded classical methods, the deep learning-based approaches rely mainly on empirical validity. Recently, Hsieh et al. \cite{hsieh2018learning} develop a promising way to learn numerical solver while providing a theoretical convergence guarantee. They demonstrate that a feed-forward CNN, which is trained to mimic a single iteration of a linear solver, can deliver faster solution than the handcrafted solver. Astonishingly for time-dependent PDEs, the temporal update step of the previously proposed neural schemes relies on an explicit forward Euler method and none of them is capable of making use of the advanced implicit and semi-implicit methods. This limitation restricts the general use of neural architectures in solving time-dependent PDEs.

In this paper, we introduce a novel neural solver for time-dependant linear PDEs. Motivated by \cite{hsieh2018learning} we construct a neural iterator from a semi-implicit update rule. We replace a single iteration of the semi-implicit scheme with a learnable parameterized function such that the fixed point of the algorithm is preserved. To leverage the theoretical guarantees we train our network in a supervised manner. Consequently, our method: (i) offers a theoretical convergence guarantee to the correct solution, (ii) converges faster over existing solvers, and (iii) generalizes to a resolutions and parameter settings very different from the ones seen at training time.
\vspace{-0.5em}
\section{Methodology}
\vspace{-0.5em}
In this section, we first introduce a general formulation on the semi-implicit method for time-dependant linear PDEs and subsequently describe our proposed neural solver.
\vspace{-0.5em}
\subsection{Background}
\vspace{-0.5em}
In the following, we consider the general IVP form of time-dependant linear PDE for the variable of interest $u$ within the computation domain $\Omega$, w.r.t. the spatial variable $x$ and temporal variable $t$, subject to Dirichlet boundary condition $b(x,t)$ at the boundary $\Gamma$
\begin{equation}
    \frac{\partial u}{\partial t} = \mathcal{F}(u,x,t;\Theta),~\forall x \in \Omega, \mbox{ s.t. } u(x,t)=b(x,t),~\forall x \in \Gamma \mbox{ and } u_{t_0}=u_0;
\end{equation}
$\mathcal{F}(u,x,t;\Theta)$ is a linear operator and can be discretized as $\sum_{i=1}^{N}\frac{\Theta_i \partial_i}{\delta x^{p_i}} u$, with uniform spatial discretization step $\delta x$ and PDE parameter set $\Theta=\{\Theta_i\}_{i=1:N}$, where $\Theta_i$ is a diagonal matrix comprising the coefficients of the differential operator $\partial_i$ of \textit{order} $p_i$. Denoting $u(x,t)$ as $u_t$ for simplicity, a first order semi-implicit update rule to get $u_{t+\delta t}$ from $u_{t}$ (with time step $\delta t$) is given by
    $\frac{u_{t+\delta t}-u_t}{\delta t} = \epsilon\mathcal{F}(u,x,t+\delta t;\Theta) + (1-\epsilon)\mathcal{F}(u,x,t;\Theta);~[0< \epsilon \leq 1].$
To obtain $u_{t+\delta t}$ one needs to solve the following linear system of equations
\begin{equation}
    \left(I-\delta t \epsilon\sum_{i=1}^{N}\frac{\Theta_i d_i}{\delta x^{p_i}}\right)u_{t+\delta t} = \delta t \epsilon\sum_{i=1}^{N}\frac{\Theta_i (\partial_i-d_i I)}{\delta x^{p_i}} u_{t+\delta t} + c(u_t, \Theta, \delta x, \delta t,\epsilon; \partial) \label{im}
\end{equation}
where $c$ is independent of $u_{t+\delta t}$ and $d_i$ is the central element of the central difference discretization of operator $\partial_i$. Note that for central difference scheme, $\partial_i-d_i I$ is real, zero-diagonal, and either circulant or skew-circulant matrix. One can use an iterative scheme to compute $u_{t+\delta t}$ from an arbitrary initialization $u_{t+\delta t}^0$ on the right-hand-side of Eq. \ref{im}. For simplicity of notation, we refer to $\left(I-\delta t \epsilon\sum_{j=1}^{N}\frac{\Theta_j d_j}{\delta x^{p_j}}\right)^{-1}\frac{\delta t \epsilon\Theta_i}{\delta x^{p_i}}$ as $\Lambda_i$, and, we drop the subscript of $u$ and use a superscript to denote a single iteration at a time $t+\delta t$. We enforce Dirichlet boundary condition using a projection step with a binary boundary mask $G$.
\begin{equation}
u^{m+1}=G\left(\sum_{i=1}^{N}\Lambda_i (\partial_i-d_i I)u^{m} + c\right)+(I-G) b \label{updaterule}
\end{equation}
Eq. \ref{updaterule} can be seen as a linear operator $u^{m+1} = \Psi(u^m)= Tu^m + k$.
We can guarantee the spectral radius of the linear transformer $T$, i.e. $\rho(T)<1$, by appropriately selecting $\delta x, \delta t, \mbox{ and } \epsilon$ [see Appendix \ref{appendixB}], leading to a \textit{fixed-point} algorithm.
\begin{figure}[t!]
\centering
\includegraphics[width=0.98\textwidth]{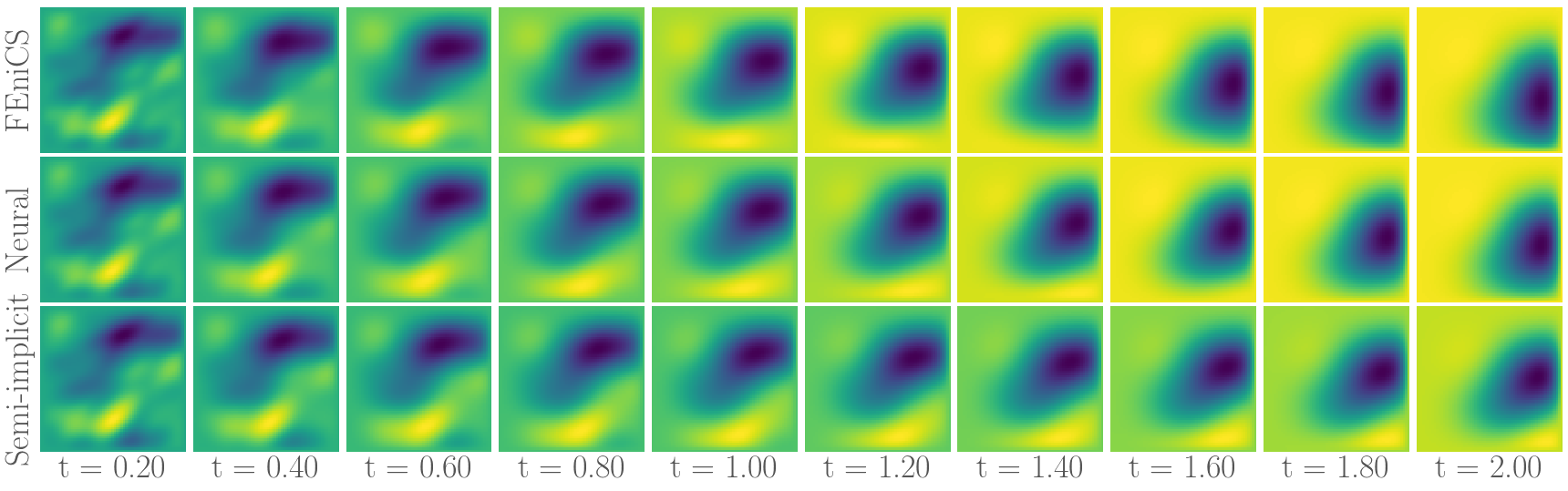}
\caption{\footnotesize Qualitative comparison of $u$ (c.f. Eq \ref{exp}) from the neural scheme (10 iterations) and a semi-implicit scheme (25 iterations) against the FEniCS solution for a test sequence of 10 time points. All methods use same initial- and boundary condition. The neural update shows consistently faster convergence than semi-implicit one.}
\label{fig2}
\vspace{-1em}
\end{figure}
\vspace{-0.5em}
\subsection{Neural Solver}
\vspace{-0.5em}
We propose the following iterator using similar notation as in \cite{hsieh2018learning}
\begin{equation} 
    \Phi_{H}(u) = \Psi(u)+G\left(\sum_{i=1}^{N}\Lambda_i  H_i w\right) \label{iterator}
\end{equation}
where $w =\Psi(u)-u$, and $H_i$ is a learned linear operator which satisfies $H_i 0=0$ for $i=1:N$. Substituting $w$ in Eq. \ref{iterator} we get $\Phi_{H}(u) = T' u+k'$, where $k'$ denotes the additive part which is 
independent  of $u$, and $T'=T+G\sum_{i=1}^{N}\Lambda_i  H_i(T-I)$.
\begin{lemma}
\label{lemma1}
For a linear PDE problem $(\{\Theta_i, \partial_i\}_{i=1:N}, G, u_0, b, \delta x, \delta t, \epsilon)$ and choice of $\{H_i\}_{i=1:N}$ if $u^{*}$ is a fixed point of $\Psi,$ it is a
fixed point of $\Phi_{H}$ in Eq. \ref{iterator}.
\end{lemma}
\begin{proof}
If $u^{*}$ satisfies $\Psi\left(u^{*}\right)=u^{*}$ then, $w=\Psi\left(u^{*}\right)-u^{*}=0$ .
Therefore the iterative rule in Eq.\ref{iterator} becomes, $\Phi_{H}\left(u^{*}\right)=\Psi\left(u^{*}\right)+G\sum_{i=1}^{N}\Lambda_i  H_i 0=u^{*} .$
\end{proof}
Additionally, if $H_i=0,~\forall i=1:N$ then $\Phi_{H}=\Psi$. Furthermore,
if $H_i=\partial_i~\forall i=1:N$ , then since $G T=T$
\begin{equation}
    \Phi_{H}(u)=\Psi(u)+G T(\Psi(u)-u)=T \Psi(u)+c=\Psi^{2}(u)
\end{equation}
which is equal to two iterations of $\Psi$. Since computing $\Phi$ requires two separate convolutions: i) $T$, and ii) $\{H_i\}_{i=1:N}$, one iteration of $\Phi_{H}$ computes same order of convolution operations as two iterations of $\Psi$. This shows that we can learn a set of $\{H_i\}_{i=1:N}$ such that our iterator $\Phi_{H}$ performs at least as good as the standard solver $\Psi$.
In the following theorem, we show that there is a convex open set of ${H_i}$ that the learning algorithm can explore.
\begin{restatable}[]{theorem}{thmone}
\label{thm1}
For fixed $G, \{\Theta_i\}_{i=1:N}, u_0, b, \delta x, \delta t$, and $\epsilon$ the spectral norm of $\Phi_{H}(u; G, \{\Theta_i\}_{i=1:N}, u_0, b, \delta x, \delta t, \epsilon)$ is a convex function of $\{H_i\}_{i=1:N}$ and the set of $\{H_i\}_{i=1:N}$ such that the spectral norm of $\Phi_{H}(u)<1$ is a convex open set.
\end{restatable}
\begin{proof}
\vspace{-0.5em}
See Appendix \ref{appendixA}.
\vspace{-0.5em}
\end{proof}
On a stark contrast with previous work \cite{hsieh2018learning}, we have several sets of parameters $\{\Theta_i\}_{i=1:N}, \delta x, \delta t$, and $\epsilon$ attached to the PDEs governing equation. Although we train on a single parameter settings, the model posses a generalization properties, which we show in the following.
\begin{restatable}[]{proposition}{propone}
\label{prop1}
For fixed \{$\partial_i\}_{i=1:N}, G$, and $\{H_i\}_{i=1:N}$, and some $u'_{0}, b', \{\Theta'_{i}\}_{i=1:N}, \delta x', \delta t'$, and $\epsilon'$, if $\Phi_{H}(u)$ is valid iterator for the PDE problem $(\{\Theta'_{i}, \partial_i\}_{i=1:N}, G, u'_{0}, b', \delta x', \delta t', \epsilon')$, then for all $u_0$ and $b,$ the iterator $\Phi_{H}(u)$ is valid iterator for the PDE problem $(\{\Theta_i, \partial_i\}_{i=1:N}, G, u_0, b, \delta x, \delta t, \epsilon)$ if $\left\|\Lambda_i\right\|< \frac{1}{\sum_{j=1}^{N}\left\|\partial_j-d_j I \right\|},~\forall i=1:N$.
\end{restatable}
\begin{proof}
\vspace{-0.5em}
See Appendix \ref{appendixA}.
\vspace{-0.5em}
\end{proof}
\section{Experiments}
\vspace{-0.5em}
We consider a 2-D advection-diffusion equation of the following form
\begin{eqnarray}
    \frac{\partial u}{\partial t} = \mathbf{v}^\top \cdot \begin{bmatrix}
\partial_{x}u\\
\partial_{y}u
\end{bmatrix} + \mathbf{D}^\top \cdot  \begin{bmatrix}
\partial_{xx}u\\
\partial_{yy}u
\end{bmatrix} \mbox{; subject to } u_{t_0} = u_0(x,y) \label{exp}
\end{eqnarray} 
where $\mathbf{v}=[v_x, v_y]^\top$ and $\mathbf{D}=[D_{xx}, D_{yy}]^\top$ are advection velocity and diffusivity respectively. We minimize the following loss function $\mathcal{L}=\frac{1}{L}\sum_{n=1}^{N}\sum_{l=1}^{L}\left\|\Phi^{n,k}_H(u^l_t)-u^l_{t+n\delta t}\right\|; k\sim\mathcal{U}[1,20]$
where $n$ is the number of time-step and $k$ iteration for a single time step is denoted as $$\Phi_H(\Phi_H \dots (\Phi_H)) = \Phi^{k}_H$$

\textbf{Data Generation: }
We consider a rectangular domain of $\Omega = [0,2\pi]\times [0,2\pi]$. Elements of $\mathbf{v}$ and $\mathbf{D}$ are drawn from a uniform distribution of $\mathcal{U}[-2.0,2.0]$ and $\mathcal{U}[0.2,0.8]$ respectively. The computational domain is discretized into 64 x 64 regular mesh. We assume zero Dirichlet boundary condition and the initial value is generated according to \cite{long2017pde} as $ u_0 = \lambda cos(kx+ly) + \gamma sin(kx+ly)$ where $\gamma$ and $ \lambda$  are drawn from a normal distribution of $\mathcal{N}(0,0.02)$, and, $k$ and $l$ are values drawn in random from a uniform distribution of $\mathcal{U}[1,9]$. We generate 200 simulations each having 50 time steps, using FEniCS \cite{logg2012automated}\cite{alnaes2015FEniCS} for $\delta t = 0.2$. Further, we take the train, test, and validation split of the simulated time series as $80\%:10\%:10\%$. A time series of a test data is shown in Fig \ref{fig2}.

\textbf{Implementation Details: }
Following \cite{hsieh2018learning}, we use a three-layer convolutional neural network to model each of the $H_i$. We use zero-padding to enforce zero Dirichlet condition at the boundary and use a kernel size of 3x3. During training, we fixed the following parameters as follows $\delta x=0.098, \delta t=0.2, \epsilon=0.9$. We use the PyTorch framework and train our network with Adam Optimizer \cite{kingma2014adam}.
\subsection{Results and Discussion}
\begin{figure}[t!]
\centering
\includegraphics[width=0.32\textwidth]{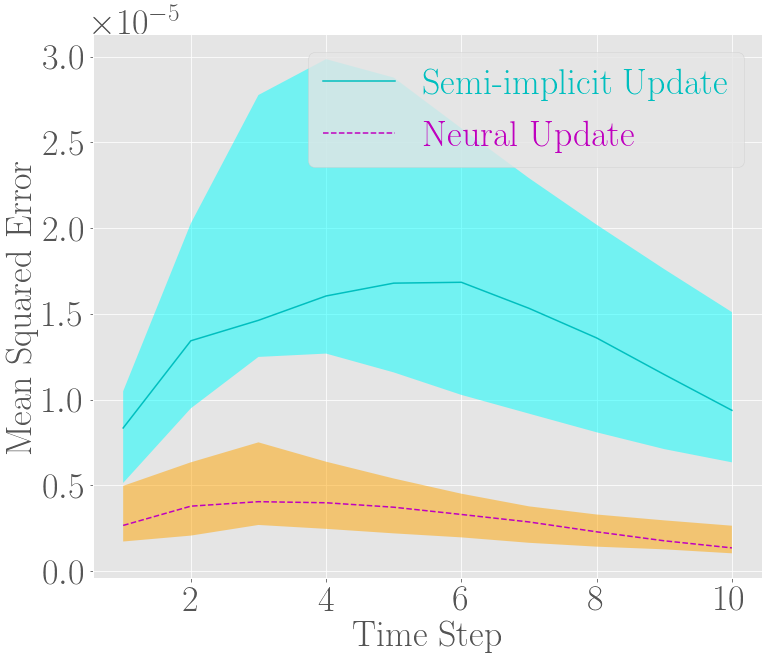}
\includegraphics[width=0.32\textwidth]{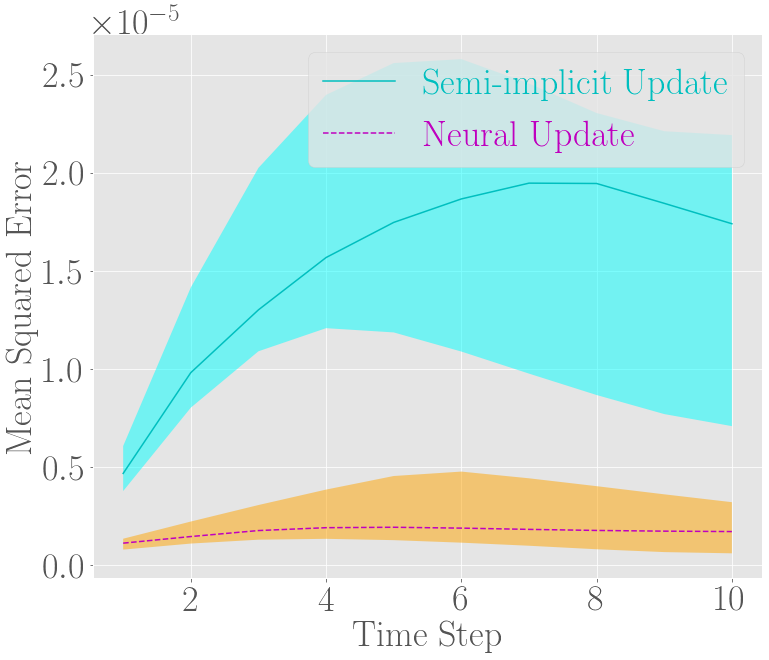}
\includegraphics[width=0.32\textwidth]{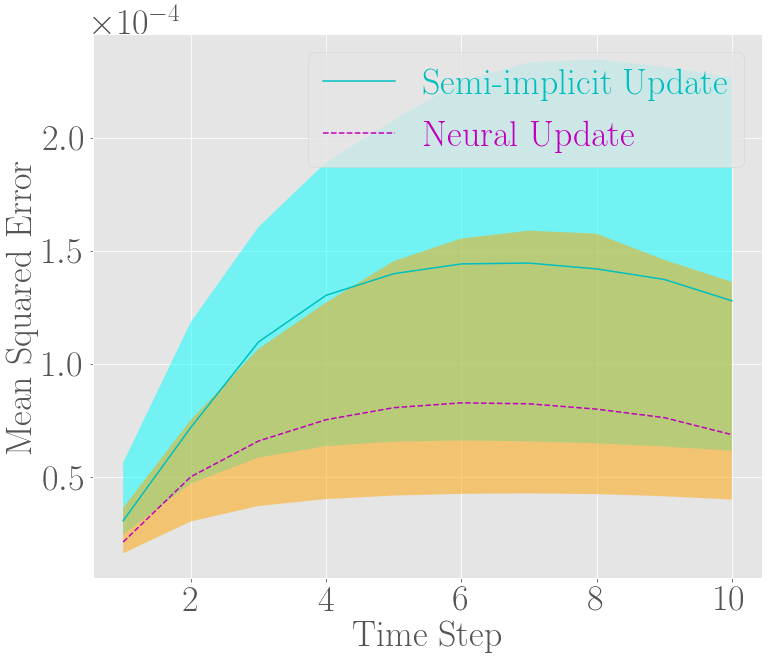} \\
(a) $\epsilon=0.75$ \hspace{8em} (b) $\delta t= 0.12$ \hspace{7em} (c) $\delta x= 0.049$
\caption{\footnotesize (a), (b), and (c) shows the mean-squared error (between FEniCS solution and Semi-implicit scheme and neural scheme) vs time plot for different parameters during test time as specified. The banded curves indicate the 25\% and 75\% percentile of the normalized errors among 20 test samples.}
\label{fig1}
\vspace{-1em}
\end{figure}

\textbf{Generalization for Different Parameters: }
We investigate the effect of different parameter settings than those used during, to validate the generalizability of the neural scheme. To study the effect of different $\epsilon$ we use the original test set. We generate two additional test cases; varying one parameter at a time : a) $\delta t=0.12$, and b) $\delta x =0.049$. The elements of $\mathbf{v}$ and $\mathbf{D}$ are drawn from the same distribution as before. The average error propagation over ten time step between the semi-implicit finite difference method and the proposed neural implicit solver is being compared in Figure \ref{fig1}.

We observe that error from the neural scheme (10 iterations per time step) is less compared to the error from the semi-implicit FDM (25 iterations per time step) scheme for all three different test sets. This affirms our hypothesis that the neural solver is more accurate compared to the semi-implicit FDM and generalizable to other parameter settings at the same time. We interpret that each $H_i$ explores an optimal discretization scheme near the manifold of $\partial_i$ by learning the sum of order rules as described in \cite{dong2017image}[c.f. Appendix \ref{appendixC}].

\textbf{Run-time Comparison: }
We compare the run time for the neural solver (10 iterations per time step) and semi-implicit scheme (25 iterations per time step), for one time-steps. The experiments are conducted on an Intel Xeon W-2123 CPU @ 3.60GHz, with code running on one of the four cores. We report that the trained neural solver takes circa 0.0148s compared to 0.0141s for the semi-implicit scheme, whereas the FEniCS solution takes 3.19s for machine precision convergence.
\vspace{-0.5em}
\section{Conclusion}
\vspace{-0.5em}
This abstract introduces a novel implicit neural scheme to solve linear time-dependent PDEs. We leverage an existing semi-implicit update rule to design a learnable iterator that provides theoretical guarantees. The learned iterator achieves faster convergence compared to the existing semi-implicit solver and produces a more accurate solution for a fixed computation budget. More importantly, we empirically demonstrate that training on a single parameter setting is enough to generalize over other parameter settings which confirms our theoretical results. The learned neural solver can be a faster alternative for simulation of various physical processes.
\newpage
\subsubsection*{Acknowledgments}
Suprosanna Shit and Ivan Ezhov are supported by the Translational Brain Imaging Training Network (TRABIT) under the European Union’s ‘Horizon 2020’ research \& innovation program (Grant agreement ID: 765148).
\bibliographystyle{unsrt}

\bibliography{neurips_2019}

\begin{thebibliography}{10}

\bibitem{raissi2019physics}
M.~Raissi et~al.
\newblock Physics-informed neural networks: A deep learning framework for
  solving forward and inverse problems involving nonlinear partial differential
  equations.
\newblock {\em Journal of Computational Physics}, 378:686 -- 707, 2019.

\bibitem{magill2018neural}
Martin Magill et~al.
\newblock Neural networks trained to solve differential equations learn general
  representations.
\newblock In {\em Advances in Neural Information Processing Systems}, pages
  4071--4081, 2018.

\bibitem{ezhov2019neural}
Ivan Ezhov et~al.
\newblock Neural parameters estimation for brain tumor growth modeling.
\newblock In {\em Proceedings of the International Conference on Medical Image
  Computing and Computer Assisted Intervention}. Springer, 2019.

\bibitem{tompson2017accelerating}
Jonathan Tompson et~al.
\newblock Accelerating {E}ulerian fluid simulation with convolutional networks.
\newblock In {\em Proceedings of the 34th International Conference on Machine
  Learning}, volume~70, pages 3424--3433. PMLR, 2017.

\bibitem{long2017pde}
Zichao Long et~al.
\newblock {PDE}-net: Learning {PDE}s from data.
\newblock In {\em Proceedings of the 35th International Conference on Machine
  Learning}, volume~80, pages 3208--3216. PMLR, 2018.

\bibitem{hsieh2018learning}
Jun-Ting Hsieh et~al.
\newblock Learning neural {PDE} solvers with convergence guarantees.
\newblock In {\em Proceedings of the International Conference on Learning
  Representations}, 2019.

\bibitem{logg2012automated}
Anders Logg et~al.
\newblock {\em Automated solution of differential equations by the finite
  element method: The {FEniCS} book}, volume~84.
\newblock Springer Science \& Business Media, 2012.

\bibitem{alnaes2015FEniCS}
Martin Aln{\ae}s et~al.
\newblock The {FEniCS} project version 1.5.
\newblock {\em Archive of Numerical Software}, 3(100), 2015.

\bibitem{kingma2014adam}
Diederik~P. Kingma and Jimmy Ba.
\newblock Adam: A method for stochastic optimization.
\newblock In {\em Proceedings of the International Conference on Learning
  Representations}, 2014.

\bibitem{dong2017image}
Bin Dong et~al.
\newblock Image restoration: Wavelet frame shrinkage, nonlinear evolution pdes,
  and beyond.
\newblock {\em Multiscale Modeling \& Simulation}, 15(1):606--660, 2017.

\bibitem{horn1991topics}
Roger Horn and Charles Johnson.
\newblock {\em Topics in matrix analysis}.
\newblock Cambridge University Press, 1991.

\end{thebibliography}
\newpage
\appendix
\section{Convergence Criteria for Semi-implicit Update}\label{appendixB}
The spectral radius of $T$ can be expresses as following
\begin{alignat*}
    \rho\rho(T) && \leq \left\|T\right\| & = \left\|G\sum_{i=1}^{N}\Lambda_i (\partial_i-d_i I)\right\|\nonumber\\
    & && \leq \left\|G\right\|\sum_{i=1}^{N}\left\|\Lambda_i\right\|\left\| (\partial_i-d_i I) \right\|;~[\mbox{ Invoking norm inequalities\cite{horn1991topics}}] \nonumber \\
    & && = \sum_{i=1}^{N}\left\|\Lambda_i\right\|\left\| (\partial_i-d_i I) \right\| ; ~[\|G\|=1] \nonumber
\end{alignat*}
Given $\left\|\Lambda_i\right\|<\frac{1}{\sum_{j=1}^{N}\left\| (\partial_j-d_j I) \right\|};~\forall i=1:N$, we have $\rho(T)<1$.

\section{Proofs}
\label{appendixA}
\thmone*
\begin{proof}
The proof is similar to Theorem 2 of \cite{hsieh2018learning}. The spectral norm $\|\cdot\|$ is convex from the sub-additive property, and $T'$ is linear in $\{H_i\}_{i=1:N}$. 
To prove that it is open, observe that $\|\cdot\|$ is a continuous function, so $(T+G\sum_{i=1}^{N}\Lambda_i  H_i(T-I))$ is continuous in $\{H_i\}_{i=1:N}$. Given $\rho(T')$, the set of $\{H_i\}_{i=1:N}$ is the preimage under this continuous function of $(0, 1 - \zeta)$ for some $\zeta > 0$, and the inverse image of open set $(0, 1 - \zeta)$ must be open.
\end{proof}
\begin{restatable}[]{lemma}{lemmatwo}
\label{lemma2}
The upper bound of the spectral norm of $\Phi_{H}$ is independent of $\{\Theta_i\}_{i=1:N}, \delta x, \delta t$, and $\epsilon$ Given $\left\|\Lambda_i\right\|< \frac{1}{\sum_{j=1}^{N}\left\|\partial_j-d_j I \right\|},~\forall i=1:N$.
\end{restatable}
\begin{proof}
Considering the spectral norm of $T'$ and invoking product and triangular inequality of norm, we obtain the following tight bound
\begin{align}
    \left\|T'\right\| &= \left\|G\sum_{i=1}^{N}\Lambda_i (\partial_i-d_i I-H_i) + G\sum_{i=1}^{N}(\Lambda_i  H_i)T \right\| \nonumber\\
    & \le \|G\|\sum_{i=1}^{N} \left\|\Lambda_i\right\| \left\|\partial_i-d_i I-H_i\right\| + \|G\|\sum_{i=1}^{N}\left\|\Lambda_i\right\|\left\| H_i\right\|\left\|T \right\| \nonumber \\
    & < \sum_{i=1}^{N} \left\|\Lambda_i\right\|\left( \left\|\partial_i-d_i I-H_i\right\| + \left\| H_i\right\|\right) ~[\|G\|=1, \|T\|<1]\nonumber
\end{align}
Given $\left\|\Lambda_i\right\|< \frac{1}{\sum_{j=1}^{N}\left\|\partial_j-d_j I \right\|},~\forall i=1:N$, we have
\begin{align}
    \left\|T'\right\| < \frac{1}{\sum_{j=1}^{N}\left\|\partial_j-d_j I \right\|}\sum_{i=1}^{N}\left( \left\|\partial_i-d_i I-H_i\right\| + \left\| H_i\right\|\right) \nonumber
\end{align}
\end{proof}
\propone*
\begin{proof}
From Theorem 1 of \cite{hsieh2018learning} and Lemma \ref{lemma1}, our iterator is valid if and only if $\rho\left(T'\right)<1$. From Lemma \ref{lemma2} the upper bound of spectral norm of iterator only depends on \{$\partial_i\}_{i=1:N}$ and $\{H_i\}_{i=1:N}$ given $\left\|\Lambda_i\right\|< \frac{1}{\sum_{j=1}^{N}\left\|\partial_j-d_j I \right\|},~\forall i=1:N$. Nonetheless, for any matrix spectral radius is upper bounded by its spectral norm. Thus, if the iterator is valid for some $u'_{0}, b', \{\Theta'_{i}\}_{i=1:N}, \delta x', \delta t'$, and $\epsilon'$, then it is valid for any feasible choice of $u_{0}, b, \{\Theta_{i}\}_{i=1:N}, \delta x, \delta t$, and $\epsilon$ satisfying the constraints.
\end{proof}
\section{Geometric Interpretation}
\label{appendixC}
Surprisingly, we find that the form of $\|T'\|$ has a special structure. As the denominator is constant the objective is to minimize $\left\|\partial_i-d_i I-H_i\right\| + \left\| H_i\right\|$ w.r.t. $\left\| H_i\right\|; \forall i=1:N$. Invoking triangular inequality of norm we have the lower-bound
$$\left\|\partial_i-d_i I-H_i\right\| + \left\| H_i\right\|\geq \left\|\partial_i-d_i I\right\|$$
Taking square of both side, we have
$$\mbox{Find } {H_i} \mbox{ such that }(\partial_i-d_i I)^\top H_i \leq \left\| H_i\right\|$$
When the equality holds the local optima is the surface of the hyper-sphere with center at $\frac{1}{2}(\partial_i-d_i I)$ with radius $\frac{1}{2}\left\|\partial_i-d_i I\right\|$. This leads us to believe that each $H_i$ explores an optimal discretization scheme near the manifold of $\partial_i$ by learning the sum of order rules as described in \cite{dong2017image}.
\end{document}